\documentclass[10pt,reqno]{amsart}
\usepackage{amsmath,amscd,amsfonts,amsthm,amssymb,latexsym,mathrsfs}
\usepackage{url,color}

\usepackage{hyperref}
\hypersetup{
    bookmarks=true,         
    unicode=false,          
    pdftoolbar=true,        
    pdfmenubar=true,        
    pdffitwindow=false,     
    pdfstartview={FitH},    
    pdftitle={My title},    
    pdfauthor={Author},     
    pdfsubject={Subject},   
    pdfcreator={Creator},   
    pdfproducer={Producer}, 
    pdfkeywords={keyword1} {key2} {key3}, 
    pdfnewwindow=true,      
    colorlinks=true,       
    linkcolor=blue,          
    citecolor=blue,        
    filecolor=blue,      
    urlcolor=blue      
    }

\theoremstyle{plain}
   \newtheorem{theorem}{Theorem}[section]
   \newtheorem{proposition}[theorem]{Proposition}
   \newtheorem{lemma}[theorem]{Lemma}
   \newtheorem{corollary}[theorem]{Corollary}
   \newtheorem{problem}{Problem}
   \newtheorem{conjecture}[theorem]{Conjecture}
   
\theoremstyle{definition}
   \newtheorem{definition}{Definition}[section]

\theoremstyle{remark}
   \newtheorem{remark}[theorem]{Remark}

\numberwithin{equation}{section}

\renewcommand{\MR}{\mathcal{M}}
\newcommand{\lma}{\lambda_{\rm max}}
\newcommand{\lmi}{\lambda_{\rm min}}

\newcommand{\zz}{\mathbf{z}}
\newcommand{\ww}{\mathbf{w}}
\newcommand{\ee}{\mathbf{e}}

\newcommand{\xx}{\mathbf{x}}
\newcommand{\yy}{\mathbf{y}}
\newcommand{\uu}{\mathbf{u}}
\renewcommand{\ss}{\mathbf{s}}
\newcommand{\vv}{\mathbf{v}}
\newcommand{\one}{\mathbf{1}}

\newcommand{\VV}{\mathbf{V}}
\newcommand{\A}{\mathcal{A}}
\newcommand{\EE}{\mathbb{E}}
\newcommand{\PP}{\mathbb{P}}

\newcommand{\ZZ}{\mathbb{Z}}

\newcommand{\RR}{\mathbb{R}}
\newcommand{\CC}{\mathbb{C}}
\newcommand{\XR}{\mathsf{X}}

\renewcommand{\Im}{{\rm Im}}

\def\newop#1{\expandafter\def\csname #1\endcsname{\mathop{\rm
#1}\nolimits}}
\newop{diag}
\newop{Spec}
\newop{Tr}
\newop{tr}
\newop{rk}
\newop{sign}

\title[Hyperbolic polynomials and the M--S--S theorem]{Hyperbolic polynomials and the Marcus--Spielman--Srivastava theorem} 

\author{Petter Br\"and\'en}
\thanks{The author is a Royal Swedish Academy of Sciences Research Fellow
  supported by a grant from the Knut and Alice Wallenberg
  Foundation. The research is also supported by the G\"oran Gustafsson Foundation.}
\address{Department of Mathematics, Royal Institute of Technology, SE-100 44 Stockholm,
Sweden}
\email{pbranden@kth.se}

\begin{document}
\begin{abstract}
Recently Marcus,  Spielman and  Srivastava gave a spectacular proof of a theorem which implies a positive solution to the Kadison--Singer problem. 
 We extend (and slightly sharpen) this theorem to the realm of hyperbolic polynomials.  A benefit of the extension is that the proof becomes coherent in its general form, and fits naturally in the theory of hyperbolic polynomials.  We also study the sharpness of the bound in the theorem, and  in the final section we describe how the hyperbolic Marcus--Spielman--Srivastava theorem may be interpreted in terms of strong Rayleigh measures. We use this to derive sufficient conditions for a weak half-plane property matroid to have $k$ disjoint bases. 
 
\end{abstract}
\maketitle

\tableofcontents

\thispagestyle{empty}

This work is based on notes from a graduate course focused on hyperbolic polynomials and the recent papers \cite{MSS1,MSS2} of Marcus,  Spielman and  Srivastava, given by the author at the Royal Institute of Technology (Stockholm) in the fall of 2013.
\newpage

\section{Introduction}

Recently Marcus,  Spielman and  Srivastava  \cite{MSS2} gave a spectacular proof of Theorem \ref{MSSmain} below, which implies a positive solution to the infamous Kadison--Singer problem  \cite{KS}. 
One purpose of this work is to extend Theorem~\ref{MSSmain}  to the realm of hyperbolic polynomials. Although our proof essentially follows the setup in \cite{MSS2}, a benefit of the extension (Theorem \ref{t1})  is that the proof becomes coherent in its general form, and fits naturally in the theory of hyperbolic polynomials.  We study the sharpness of the bound in Theorem \ref{t1}. We prove that a conjecture in \cite{MSS2} on the sharpness of the bound (Conjecture \ref{maxmax} in this paper) is equivalent to the seemingly weaker Conjecture~\ref{maxmax2}.  
Using known results about the asymptotic behavior of the largest zero of Jacobi polynomials, we prove in Section~\ref{sbound} that the bound is close to being optimal in the hyperbolic setting, see Proposition~\ref{lowprop}. 

In the final section we describe how Theorem \ref{t1} may be interpreted in terms of strong Rayleigh measures. We use this to derive sufficient conditions for a weak half-plane property matroid to have $k$ disjoint bases. These conditions are very different from Edmonds characterization in terms of the rank function of the matroid \cite{Edm}.

\medskip

The following theorem is a stronger version of Weaver's $KS_k$ conjecture \cite{We} which is known to imply a positive solution to the Kadison--Singer problem \cite{KS}. See \cite{Cas} for a review of the many consequences of Theorem \ref{MSSmain}.

\begin{theorem}[Marcus,  Spielman and  Srivastava \cite{MSS2}]\label{MSSmain}
Let $k \geq 2$ be an integer. Suppose $\vv_1, \ldots, \vv_m \in \CC^d$ satisfy $\sum_{i=1}^m \vv_i\vv_i^* = I$, where $I$ is the identity matrix. If  $\|\vv_i \|^2 \leq \epsilon$ for all $1\leq i \leq m$, then  there is a partition of $S_1\cup S_2 \cup \cdots \cup S_k=[m]:=\{1,2,\ldots,m\}$ such that 
\begin{equation}\label{sqbound1}
\left\| \sum_{i \in S_j} \vv_i\vv_i^*  \right\| \leq \frac {(1+\sqrt{k\epsilon})^2} k,    
\end{equation}
for each $j \in [k]$, where $\|\cdot \|$ denotes the operator matrix norm. 
\end{theorem}

Hyperbolic polynomials are multivariate generalizations of real--rooted polynomials, which have their  origin  in PDE theory where they were studied by Petrovsky, G\aa rding, Bott, Atiyah and H\"ormander, see \cite{ABG,Ga,Horm}.  During recent years  hyperbolic polynomials have been studied in diverse areas such as 
control theory, optimization, real algebraic geometry, probability theory, computer science and combinatorics, see  \cite{Pem,Ren,Vin,Wag} and the references therein. 

A homogeneous polynomial $h(\xx) \in \RR[x_1, \ldots, x_n]$ is \emph{hyperbolic} with respect to a vector 
$\ee \in \RR^n$ if $h(\ee) \neq 0$, and if for all $\xx \in \RR^n$ the univariate polynomial $t \mapsto h(t\ee -\xx)$ has only real 
zeros. Here are some examples of hyperbolic polynomials:
\begin{enumerate}
\item Let $h(\xx)= x_1\cdots x_n$. Then $h(\xx)$ is hyperbolic with respect to any vector $\ee \in \RR_{++}^n=(0,\infty)^n$: 
$$
h(t\ee-\xx) = \prod_{j=1}^n (te_j-x_j).
$$
\item Let $X=(x_{ij})_{i,j=1}^n$ be a matrix of $n(n+1)/2$ variables where we impose $x_{ij}=x_{ji}$. Then $\det(X)$ is hyperbolic with respect to 
$I=\diag(1, \ldots, 1)$. Indeed $t \mapsto \det(tI-X)$ is the characteristic polynomial of the symmetric matrix $X$, so it has only real zeros.

More generally we may consider complex hermitian $Z=(x_{jk}+iy_{jk})_{j,k=1}^n$ (where $i = \sqrt{-1}$) of $n^2$ real variables where we impose $x_{jk}=x_{kj}$ and $y_{jk}=-y_{kj}$, for all $1\leq j,k \leq n$.  Then $\det(Z)$ is a real polynomial which is hyperbolic with respect to $I$.
\item Let $h(\xx)=x_1^2-x_2^2-\cdots-x_n^2$. Then $h$ is hyperbolic with respect to $(1,0,\ldots,0)^T$. 
\end{enumerate}

Suppose $h$ is hyperbolic with respect to $\ee \in \RR^n$. We may write 
\begin{equation}\label{dalambdas}
h(t\ee-\xx) = h(\ee)\prod_{j=1}^d (t - \lambda_j(\xx)),
\end{equation}
where $\lma(\xx)=\lambda_1(\xx) \geq \cdots \geq \lambda_d(\xx)=\lmi(\xx)$ are called the \emph{eigenvalues} of $\xx$ (with respect to $\ee$), and $d$ is the degree of $h$. In particular 
\begin{equation}\label{prolambda}
h(\xx) = h(\ee)\lambda_1(\xx) \cdots \lambda_d(\xx). 
\end{equation}

By homogeneity 
\begin{equation}\label{dilambdas}
\lambda_j(s\xx+t\ee)= 
\begin{cases}
s\lambda_j(\xx)+t &\mbox{ if } s\geq 0 \mbox{ and } \\
s\lambda_{d-j}(\xx)+t &\mbox{ if } s \leq 0
\end{cases},
\end{equation}
for all $s,t \in \RR$ and $\xx \in \RR^n$.

The (open) \emph{hyperbolicity cone} is the set 
$$
\Lambda_{\tiny{++}}= \Lambda_{\tiny{++}}(\ee)= \{ \xx \in \RR^n : \lmi(\xx) >0\}.
$$
We denote its closure by $\Lambda_{\tiny{+}}= \Lambda_{\tiny{+}}(\ee)=\{ \xx \in \RR^n : \lmi(\xx) \geq 0\}$. 
Since $h(t\ee-\ee)=h(\ee)(t-1)^d$ we see that $\ee \in \Lambda_{\tiny{++}}$. The hyperbolicity cones for the examples  above are:
\begin{enumerate}
\item $\Lambda_{\tiny{++}}(\ee)= \RR_{++}^n$. 
\item $\Lambda_{\tiny{++}}(I)$ is the cone of positive definite matrices.
\item $\Lambda_{\tiny{++}}(1,0,\ldots,0)$  is the \emph{Lorentz cone} 
$$
\left\{\xx \in \RR^n : x_1 > \sqrt{x_2^2+\cdots+x_n^2}\right\}.
$$
\end{enumerate}

The following theorem collects a few fundamental facts about hyperbolic polynomials and their hyperbolicity cones. For proofs see \cite{Ga,Ren}. 

\begin{theorem}[G\aa rding, \cite{Ga}]\label{hypfund}
Suppose $h$ is hyperbolic with respect to $\ee \in \RR^n$. 
\begin{enumerate}
\item $\Lambda_+(\ee)$ and $\Lambda_{++}(\ee)$ are convex cones.
\item $\Lambda_{++}(\ee)$ is the connected component of 
$$
\{ \xx \in \RR^n : h(\xx) \neq 0\}
$$
which contains $\ee$. 
\item $\lmi : \RR^n \rightarrow \RR$ is a concave function, and  $\lma : \RR^n \rightarrow \RR$ is a convex function. 
\item If $\ee' \in \Lambda_{++}(\ee)$, then $h$ is hyperbolic with respect to $\ee'$ and $\Lambda_{++}(\ee')=\Lambda_{++}(\ee)$.
\end{enumerate}
\end{theorem} 

Recall that the \emph{lineality space}, $L(C)$, of a convex cone $C$ is $C \cap -C$, i.e., the largest linear space contained in $C$. It follows that $L(\Lambda_+)= \{\xx : \lambda_i(\xx)=0 \mbox{ for all } i\}$, see e.g. \cite{Ren}.

The \emph{trace}, \emph{rank} and \emph{spectral radius} (with respect to $\ee$) of  $\xx \in \RR^n$ are defined as for matrices: 
$$
\tr(\xx) = \sum_{i=1}^d\lambda_i(\xx), \ \ \rk(\xx)= \#\{ i : \lambda_i(\xx)\neq 0\} \ \  \mbox{ and } \ \  \|\xx\| = \max_{1\leq i\leq d} |\lambda_i(\xx)|.
$$
Note that $\| \xx \| = \max\{ \lma(\xx), -\lmi(\xx)\}$ and hence $ \| \cdot \|$ is convex by Theorem~\ref{hypfund}~(3). It follows that $\| \cdot \|$ is a seminorm and that $\| \xx \|=0$ if and only if $\xx \in L(\Lambda_+)$. Hence $\| \cdot \|$ is a norm if and only if  $L(\Lambda_+)=\{0\}$. 

The following theorem is a generalization of Theorem \ref{MSSmain} to hyperbolic polynomials. 

\begin{theorem}\label{t1}
Let $k\geq 2$ be an integer and $\epsilon$ a positive real number. Suppose $h$ is hyperbolic  with respect to $\ee \in \RR^n$, and let $\uu_1, \ldots, \uu_m \in \Lambda_{+}$ be such that 
\begin{itemize}
\item[] $\rk(\uu_i) \leq 1$ for all $1\leq i \leq m$,
\item[] $\tr(\uu_i) \leq \epsilon$ for all $1\leq i \leq m$, and 
\item[] $\uu_1+ \uu_2+\cdots+ \uu_m=\ee$. 
\end{itemize}
Then there is a partition of $S_1\cup S_2 \cup \cdots \cup S_k=[m]$ such that 
\begin{equation}\label{sqbound}
\left\| \sum_{i \in S_j} \uu_i  \right\| \leq \frac 1 k \delta(k\epsilon, m),    
\end{equation}
for each $j \in [k]$, where 
$$
\delta(\alpha, m):=\left( 1-\frac 1 m +\sqrt{\alpha - \frac 1 m \left(1-\frac 1 m\right)}\right)^2.
$$
\end{theorem}
We recover (a slightly improved) Theorem \ref{MSSmain} when $h= \det$ in Theorem \ref{t1}.

\section{Compatible families of polynomials}

Let $f$ and $g$ be two real--rooted polynomials of degree $n-1$ and $n$, respectively. We say that $f$ is an \emph{interleaver} of $g$ if 
$$
\beta_1 \leq \alpha_1\leq \beta_2 \leq \alpha_2 \leq \cdots \leq \alpha_{n-1} \leq \beta_n, $$ 
where $\alpha_1 \leq \cdots \leq \alpha_{n-1}$ and $\beta_1 \leq \cdots \leq \beta_{n}$ are the zeros of $f$ and $g$, respectively. 

A family of polynomials $\{f_1(x), \ldots, f_m(x)\}$ of real--rooted polynomials of the same degree and the same sign of leading coefficients  is called \emph{compatible} if it satisfies any of the equivalent conditions in the next theorem. Theorem \ref{CS} has been discovered several times. We refer to \cite[Theorem 3.6]{CS} for a proof. 
\begin{theorem}\label{CS}
Let $f_1(x), \ldots, f_m(x)$ be real--rooted polynomials of the same degree and with positive leading coefficients. The following are equivalent. 
\begin{enumerate}
\item $f_1(x), \ldots, f_m(x)$ have a common interleaver. 
\item for all $p_1, \ldots, p_m \geq 0$, $\sum_{i}p_i=1$, the polynomial
$$
p_1f_1(x)+ \cdots+ p_mf_m(x)
$$
is real--rooted. 
\end{enumerate}
\end{theorem}

\begin{lemma}[\cite{MSS1}]\label{largestz}
Let $f_1,\ldots, f_m$ be real--rooted polynomials 
that have the same degree and positive leading coefficients, and suppose $p_1, \ldots, p_m \geq 0$ sum to one.  If $\{f_1,\ldots, f_m\}$ is compatible, then for some $1 \leq i \leq m$ with $p_i >0$ the largest zero of $f_i$ is smaller or equal to the largest zero of the polynomial 
$$
 f=p_1f_1 + p_2f_2 + \cdots + p_mf_m.
$$\end{lemma}
\begin{proof}
If $\alpha$ is the largest zero of the common interleaver, then $f_i(\alpha) \leq 0$ for all $i$, so that the largest zero, $\beta$,  of $f(x)$ is located in the interval $[\alpha, \infty)$, as are the largest zeros of $f_i$ for each $1\leq i \leq m$. Since $f(\beta)=0$, there is an index $i$ with $p_i >0$ such that $f_i(\beta) \geq 0$. Hence the largest zero of $f_i$ is at most $\beta$. 
\end{proof}

\begin{definition}
Suppose $S_1, \ldots, S_m$ are finite sets. A family of polynomials, $\{f(\ss;t)\}_{\ss \in S_1 \times \cdots \times S_m}$, for which all non-zero members are of the same degree and have the same signs of their leading coefficients
 is called \emph{compatible} 
if for all choices of independent random variables 
$\XR_1 \in S_1, \ldots, \XR_m \in S_m$, the polynomial 
$
\EE f(\XR_1,\ldots, \XR_n;t) 
$
is real--rooted. 
\end{definition}

The notion of compatible families of polynomials is less general than that of \emph{interlacing families of polynomials} in \cite{MSS1,MSS2}. However since all families appearing here (and in \cite{MSS1,MSS2}) are compatible we find it more convenient to work with these. The following theorem is in essence from \cite{MSS1}. 

\begin{theorem}\label{expfam}
Let $\{f(\ss;t)\}_{\ss \in S_1 \times \cdots \times S_m}$ be a compatible family, and let \ $\XR_1 \in S_1, \ldots, \XR_m \in S_m$ be independent random variables such that $\EE f(\XR_1,\ldots, \XR_m;t) \not \equiv 0$. Then there is a tuple $\ss=(s_1, \ldots, s_n) \in S_1 \times \cdots \times S_m$, with $\PP[\XR_i=s_i]>0$ for each $1\leq i \leq m$, such that the largest zero of $f(s_1,\ldots, s_m;t)$ is smaller or equal to the largest zero of 
$\EE f(\XR_1,\ldots, \XR_m;t)$. 
\end{theorem}

\begin{proof}
The proof is by induction over $m$. The case when $m=1$ is Lemma \ref{largestz}, so suppose $m>1$.  If 
$S_m=\{c_1,\ldots, c_k\}$, then  
$$
\EE f(\XR_1,\ldots, \XR_m;t)= \sum_{i=1}^k q_i \EE f(\XR_1,\ldots, \XR_{m-1}, c_i;t), 
$$
for some $q_i \geq 0$. However 
$$
\sum_{i=1}^k p_i \EE f(\XR_1,\ldots, \XR_{m-1}, c_i;t)
$$
is real--rooted for all choices of $p_i \geq 0$ such that $\sum_ip_i=1$. By Lemma \ref{largestz} and Theorem \ref{CS} there is an index $j$ with $q_j>0$ such that $\EE f(\XR_1,\ldots, \XR_{m-1}, c_j;t) \not \equiv 0$ and such that the largest zero of $\EE f(\XR_1,\ldots, \XR_{m-1}, c_j;t)$ is no larger than the largest zero of $\EE f(\XR_1,\ldots, \XR_m;t)$. The theorem now follows by induction. 
 \end{proof}

 \section{Mixed hyperbolic polynomials}
 
 Recall that the \emph{directional derivative} of $h(\xx) \in  \RR[x_1,\ldots, x_n]$ with respect to $\vv =(v_1,\ldots, v_n)^T \in \RR^n$ is defined as 
 $$
 D_\vv h(\xx) := \sum_{k=0}^n v_k \frac{ \partial h }{\partial x_k}(\xx),
 $$
 and note that 
 \begin{equation}\label{dvalt}
  (D_\vv h)(\xx+t\vv) = \frac d {dt} h(\xx+ t \vv) . 
\end{equation}
If $h$ is hyperbolic with respect to $\ee$, then 
$$
\tr(\vv)= \frac {D_\vv h(\ee)}{h(\ee)},
$$
by \eqref{dalambdas}. Hence $\vv \rightarrow \tr(\vv)$ is linear.

The following theorem is essentially known, see e.g. \cite{BGLS,Ga,Ren}. However we need slightly more general results, so we provide proofs below, when necessary. 
\begin{theorem}\label{direct}
Let $h$ be a hyperbolic polynomial and let $\vv \in \Lambda_+$ be such that   $D_\vv h \not \equiv 0$. Then 
\begin{enumerate}
\item $D_\vv h$ is hyperbolic with hyperbolicity cone containing $\Lambda_{++}$.  
\item The polynomial $h(\xx)-yD_\vv h(\xx) \in \RR[x_1,\ldots, x_n,y]$ is hyperbolic with hyperbolicity cone containing $\Lambda_{++} \times \{y: y \leq 0\}$. 
\item The rational function 
$$
\xx \mapsto \frac {h(\xx)}{D_\vv h(\xx)}
$$
is concave on $\Lambda_{++}$. 
\end{enumerate}
\end{theorem}
\begin{proof}
(1). See \cite[Lemma 4]{BrOp}.

(2). The polynomial $h(\xx)y$ is hyperbolic with hyperbolicity cone containing $\Lambda_{++} \times \{y : y<0\}$. Hence so is $H(\xx,y):= D_{(\vv,-1)} h(\xx)y= h(\xx)- y D_\vv h(\xx)$ by (1). Since $H(\ee',0) = h(\ee') \neq 0$ for each $\ee' \in \Lambda_{++}$, we see that also $\Lambda_{++}\times \{0\}$ is a subset of the hyperbolicity cone (by Theorem \ref{hypfund} (2)) of $H$.

(3). If $\xx \in \Lambda_{++}$, then (by Theorem \ref{hypfund} (2)) $(\xx,y)$ is in the closure of the hyperbolicity cone of $H(\xx,y)$ if and only if 
$$
y \leq \frac {h(\xx)}{D_\vv h(\xx)}.
$$
Since hyperbolicity cones are convex  
$$
y_1 \leq \frac {h(\xx_1)}{D_\vv h(\xx_1)} \mbox{ and } y_2 \leq \frac {h(\xx_2)}{D_\vv h(\xx_2)} \mbox{ imply } y_1+y_2 \leq \frac {h(\xx_1+\xx_2)}{D_\vv h(\xx_1+\xx_2)}, 
$$
for all $\xx_1,\xx_2 \in \Lambda_{++}$, 
from which (3) follows. 
\end{proof}

\begin{lemma}\label{rankalt}
Let $h$ be hyperbolic with hyperbolicity cone $\Lambda_{++}\subseteq \RR^n$. The rank function does not depend on the choice of $\ee \in \Lambda_{++}$, and 
$$
\rk(\vv)= \max\{ k : D_\vv^kh \not \equiv 0\}, \quad \mbox{ for all } \vv \in \RR^n.
$$
\end{lemma}
\begin{proof}
That the rank does not depend on the choice of $\ee \in \Lambda_{++}$ is known, see \cite[Prop. 22]{Ren} or \cite[Lemma 4.4]{BrObs}.

By \eqref{dvalt}
\begin{equation}\label{mag}
h(\xx-y\vv) = \left( \sum_{k=0}^{\infty} \frac {(-y)^k D_\vv^k}{k!} \right) h(\xx). 
\end{equation}
 Thus  
$$
h(\ee-t\vv) = h(\ee)\prod_{j=1}^d(1-t\lambda_j(\vv))= \sum_{k=0}^d (-1)^k\frac {D^k_\vv h(\ee)} {k!} t^k, 
$$
 and hence $\rk(\vv)= \deg h(\ee-t\vv)=  \max\{k : D^k_\vv h(\ee)\neq 0\}$. Since the rank does not depend on the choice of $\ee \in \Lambda_{++}$, if $D^{k+1}_\vv h(\ee)=D^{k+2}_\vv h(\ee)=\cdots =0$ for some 
$\ee \in \Lambda_{++}$, then $D^{k+1}_\vv h(\ee')=D^{k+2}_\vv h(\ee')=\cdots =0$  for all $\ee' \in \Lambda_{++}$. Since $ \Lambda_{++}$ has non-empty interior this means $D^{k+1}_\vv h \equiv 0$.
\end{proof}

If $h(\xx) \in \RR[x_1,\ldots, x_n]$ and $\vv_1, \ldots, \vv_m \in \RR^n$ let $h[\vv_1, \ldots, \vv_m]$ be the polynomial in $\RR[x_1,\ldots, x_n,y_1,\ldots, y_m]$ defined by 
$$
h[\vv_1, \ldots, \vv_m] = \prod_{j=1}^m \left(1-y_jD_{\vv_j}\right) h(\xx).
$$
By iterating Theorem \ref{direct} (2) we get: 
\begin{theorem}\label{mixhyp}
If $h(\xx)$ is hyperbolic with hyperbolicity cone $\Lambda_{++}$ and $\vv_1, \ldots, \vv_m \in \Lambda_+$, then $h[\vv_1, \ldots, \vv_m]$ is hyperbolic with hyperbolicity cone containing $\Lambda_{++} \times (-\RR_+^m)$, where $\RR_+ := [0,\infty)$. 
\end{theorem}

\begin{lemma}\label{rk1le}
Suppose $h$ is hyperbolic. If $\vv_1, \ldots, \vv_m \in \Lambda_+$ have rank at most one, then 
$$
h[\vv_1, \ldots, \vv_m] = h(\xx-y_1\vv_1 - \cdots - y_m \vv_m).
$$
\end{lemma}

\begin{proof}
If $\vv$ has rank at most one, then $D_\vv^k h \equiv 0$ for all $k \geq 2$ by Lemma \ref{rankalt}. Hence, by \eqref{mag}, 
$$
h(\xx-y\vv) = \left( \sum_{k=0}^{\infty} \frac {(-y)^k D_\vv^k}{k!} \right) h(\xx)= (1-yD_{\vv})h(\xx),
$$
from which the lemma follows.
\end{proof}

Note that $(\vv_1,\ldots,\vv_m) \mapsto h[\vv_1,\ldots,\vv_m]$ is affine linear in each coordinate, i.e., for all $p \in \RR$ and $1\leq i \leq m$:
\begin{align*}
& h[\vv_1,\ldots,(1-p)\vv_i+p\vv_i',\ldots, \vv_m] \\
= &(1-p)h[\vv_1,\ldots,\vv_i,\ldots, \vv_m] +ph[\vv_1,\ldots,\vv_i',\ldots, \vv_m].
\end{align*}
Hence if $\XR_1, \ldots, \XR_m$ are  independent random variables in $\RR^n$, then 
\begin{equation}\label{mixedexp}
\EE h[\XR_1,\ldots,\XR_m] = h[\EE \XR_1,\ldots,\EE \XR_m]. 
\end{equation}

\begin{theorem}\label{mixedchar}
Let $h(\xx)$ be hyperbolic with respect to $\ee\in \RR^n$, let $V_1, \ldots, V_m$ be finite sets of vectors  in $\Lambda_+$, and let $\ww \in \RR^{n+m}$.  For $\VV =(\vv_1,\ldots, \vv_m) \in  V_1\times \cdots \times V_m$, let 
$$f(\VV;t) := h[\vv_1,\ldots, \vv_m](t\ee +\ww).
$$ 
Then $\{f(\VV;t)\}_{\VV \in V_1\times \cdots \times V_m}$ is a compatible family.

In particular if in addition all vectors in $V_1 \cup \cdots \cup V_m$ have rank at most one, and 
$$
g(\VV;t) := h(t\ee +\ww- \alpha_1\vv_1-\cdots-\alpha_m\vv_m),
$$
where $\ww \in \RR^n$  and $(\alpha_1,\ldots, \alpha_m)\in \RR^m$, then $\{g(\VV;t)\}_{\VV \in V_1\times \cdots \times V_m}$ is a compatible family.
\end{theorem}

\begin{proof}
Let $\XR_1 \in V_1, \ldots, \XR_m \in V_m$ be independent random variables. Then the polynomial 
$\EE h[\XR_1, \ldots, \XR_m]= h[\EE \XR_1,\ldots,\EE \XR_m]$ is hyperbolic with respect to $(\ee, 0,\ldots,0)$ by Theorem \ref{mixhyp} (since $\EE \vv_i \in \Lambda_+$ for all $i$ by convexity). In particular the polynomial $\EE f(\XR_1,\ldots, \XR_m;t)$ is real--rooted.

 The second assertion is an immediate consequence of the first combined with Lemma \ref{rk1le}.

\end{proof}

 \section{Bounds on zeros of mixed characteristic polynomials}
 To prove Theorem \ref{hypprob}, we want to bound the zeros of the \emph{mixed characteristic polynomial}
\begin{equation}\label{mip}
t \mapsto h[\vv_1, \ldots, \vv_m](t\ee+\one), 
\end{equation}
 where $h$ is hyperbolic with respect to $\ee \in \RR^n$, $\one \in \RR^m$ is the all ones vector, and $\vv_1, \ldots, \vv_m \in \Lambda_+(\ee)$ satisfy $\vv_1+\cdots+\vv_m =\ee$ and $\tr(\vv_i) \leq \epsilon$ for all $1\leq i \leq m$. 
 
 \begin{remark}\label{hypid}
 Note that a real number $\rho$ is larger than the maximum zero of \eqref{mip} if and only if $\rho \ee +\one$ is in the hyperbolicity cone $\Gamma_{++}$ of $h[\vv_1, \ldots, \vv_m]$. Hence the maximal zero of \eqref{mip} is equal to 
 $$
 \inf \{ \rho >0 :  \rho \ee +\one \in \Gamma_{++}\}.
 $$
 \end{remark}

 For the remainder of this section, let $h \in \RR[x_1,\ldots, x_n]$ be hyperbolic with respect to $\ee$, and let $\vv_1,\ldots, \vv_m \in \Lambda_{++}$. To enhance readability in the computations to come, 
let $\partial_j := D_{\vv_j}$ and 
$$
\xi_j[g] := \frac {g}{\partial_j g}.
$$

Note that a continuously differentiable concave function $f : (0,\infty) \to \RR$ satisfies 
$$
f(t+\delta) \geq f(t)+ \delta f'(t+\delta), \quad \mbox{ for all } \delta \geq 0.
$$
Hence  by Theorem \ref{direct}
\begin{equation}\label{concon}
\xi_i[h](\xx+\delta \vv_j) \geq \xi_i[h](\xx) + \delta \partial_j \xi_i[h](\xx+\delta \vv_j)
\end{equation}
for all $\xx \in \Lambda_{+}$ and $\delta \geq 0$. The following elementary identity is left for the reader to verify. 
\begin{lemma}\label{tech}
$$\xi_i[h-\partial_jh]=\xi_i[h]-\frac {\partial_j \xi_i[h]\cdot \xi_j[\partial_ih]}{\xi_j[\partial_ih]-1}.$$
\end{lemma}
\begin{lemma}\label{engine}
If $\xx \in \Lambda_{++}$, $1\leq i,j \leq n$, $\delta > 1$ and 
$$
\xi_j[h](\xx)  \geq \frac \delta {\delta-1}, 
$$
then 
$$
\xi_i[h-\partial_jh](\xx+\delta \vv_j) \geq \xi_i[h](\xx).
$$
\end{lemma}

\begin{proof}
Since $\xi_i[h]$ is concave on $\Lambda_{++}$ (Theorem~\ref{direct} (3)) and homogeneous of degree one:
$$
\frac {\xi_i[h](\zz+\delta \vv_j) - \xi_i[h](\zz)}{\delta} \geq \xi_i[h](\vv_j), \ \ \ \mbox{ for all } \zz \in \Lambda_{++}.
$$
Hence
\begin{equation}\label{parat}
\partial_j \xi_i[h](\zz) \geq  \xi_i[h](\vv_j)\geq 0, \ \ \ \mbox{ for all } \zz \in \Lambda_{++}.
\end{equation}
If $\zz \in \Lambda_{++}$, then (by Theorem \ref{hypfund} (2)) $(\zz, t)$ is in the closure of the hyperbolicity cone of $h-y\partial_j h$ if and only if $t \leq \xi_j[h](\zz)$. By Theorem \ref{direct} the polynomial 
$$
D_{(\vv_i,0)}(h-y\partial_j h) = \partial_i h -y \partial_j \partial_i h 
$$
is hyperbolic with hyperbolicity cone containing the hyperbolicity cone of $h-y\partial_j h$. Hence if $\zz \in \Lambda_{++}$ and  $t  \leq \xi_j[h](\zz)$, then $t  \leq \xi_j[\partial_ih](\zz)$, and thus 
\begin{equation}\label{parata}
\xi_j[\partial_ih](\zz) \geq \xi_j[h](\zz), \ \ \ \mbox{ for all } \zz \in \Lambda_{++}.
\end{equation}

Let $\xx$ be as in the statement of the lemma. By Lemma \ref{tech} and \eqref{concon}
\begin{align*}
\xi_i[h-\partial_jh](\xx+\delta \vv_j) - \xi_i[h](\xx) &= \xi_i[h](\xx+\delta \vv_j)-\xi_i[h](\xx)- \frac {\partial_j \xi_i[h]\cdot \xi_j[\partial_ih]}{\xi_j[\partial_ih]-1}(\xx+\delta \vv_j)\\
&\geq \partial_j \xi_i[h](\xx+\delta \vv_j) \left( \delta - \frac { \xi_j[\partial_ih](\xx+\delta \vv_j) }{\xi_j[\partial_ih](\xx+\delta \vv_j)-1}\right)\\
&\geq \xi_i[h](\vv_j) \left( \delta - \frac { \delta/(\delta-1)}{ \delta/(\delta-1)-1}\right) =0, 
\end{align*}
where the last inequality follows from \eqref{parat}, \eqref{parata} and the concavity of $\zz \rightarrow \xi_j[h](\zz)$. 

\end{proof}

Consider $\RR^{n+m}=\RR^n\oplus \RR^m$ and let $\ee_1,\ldots, \ee_m$ be the standard bases of $\RR^m$ (inside $\RR^n\oplus \RR^m$). 

\begin{corollary}\label{corbond}
Suppose $h$ is hyperbolic with respect to $\ee \in \RR^n$, and let $\Gamma_+$ be the (closed) hyperbolicity cone of $h[\vv_1,\ldots, \vv_m]$, where $\vv_1,\ldots, \vv_m \in \Lambda_{+}(\ee)$. 
Suppose $t_i, t_j > 1$ and $\xx \in \Lambda_{+}(\ee)$ are such that 
$$
\xx+t_k \ee_k \in \Gamma_+, \quad \mbox{ for } k \in \{i,j\}. 
$$
Then 
$$
\xx+\frac {t_j}{t_j-1}\vv_j + \ee_j + t_i \ee_i \in \Gamma_+.  
$$ 
Moreover if $\xx+t_k \ee_k \in \Gamma_+$ for all $k \in [m]$, then
$$
\xx+ \left(1-\frac 1 m\right) \sum_{i=1}^m \frac {t_i}{t_i-1}\vv_i +\left(1-\frac 1 m\right)\sum_{i=1}^m \ee_i+ \frac 1 m\sum_{i=1}^m t_i\ee_i \in \Gamma_+. 
$$
\end{corollary}
\begin{proof}
By continuity we may assume $\xx,\vv_1,\ldots, \vv_m \in \Lambda_{++}(\ee)$. 
Let $\delta_k = t_k/(t_k-1)$. Then 
$$ 
\xx + t_k \ee_k \in \Gamma_+ \mbox{ if and only if } \xi_k[h] \geq \frac {\delta_k}{\delta_k-1}.
$$
Also $\xi_i[h-\partial_jh](\xx+\delta_j \vv_j) \geq \delta_i/(\delta_i-1)$ is equivalent to  
$$
\xx+\delta_j\vv_j + \ee_j+\frac {\delta_i} {\delta_i-1} \ee_i \in \Gamma_+.
$$
Hence the first part follows from Lemma \ref{engine}. 

Suppose $\xx+t_k \ee_k \in \Gamma_+$ for all $k \in [m]$. Since $\xx+s\ee_1,  \vv_1 \in \Gamma_+$ for all $s \leq t_1$, the vector 
$$
\xx' := \xx+\frac {t_1}{t_1-1}\vv_1 + \ee_1
$$
is in the hyperbolicity cone of $(1-y_1D_{\vv_1})h$. 
By the first part we have $\xx'+t_2\ee_2, \xx'+t_3\ee_3\in \Gamma_+$. Hence we may apply the first part of the theorem with $h$ replaced by $(1-y_1D_{\vv_1})h$ to conclude 
$$
\xx'+ \frac {t_2}{t_2-1}\vv_2 + \ee_2+ t_3\ee_3=\xx+\frac {t_1}{t_1-1}\vv_1 + \frac {t_2}{t_2-1}\vv_2+\ee_1 +\ee_2 + t_3\ee_3\in \Gamma_+. 
$$
By continuing this procedure with different orderings we may conclude that 
$$
\xx +  \left(\sum_{i=1}^m \frac {t_i}{t_i-1}\vv_i\right)-\frac {t_j}{t_j-1}\vv_j  +\left(\sum_{i=1}^m \ee_i\right)-\ee_j+t_j\ee_j \in \Gamma_+,
$$
for each $1\leq j \leq m$. The second part now follows from convexity of $\Gamma_+$ upon taking the convex sum of these vectors.
\end{proof}

\begin{theorem}\label{mainbound}
Suppose $h$ is hyperbolic with respect to $\ee \in \RR^n$ and suppose $\vv_1,\ldots, \vv_m \in \Lambda_{+}(\ee)$ are such that $\ee = \vv_1+\cdots+\vv_m$, where $\tr(\vv_j) \leq \epsilon$ for each $1\leq j \leq m$. Then the largest zero of the polynomial 
$$
t \mapsto h[\vv_1, \ldots, \vv_m](t\ee+\one)
$$
is at most 
$$
\delta(\epsilon, m):=\left( 1-\frac 1 m +\sqrt{\epsilon - \frac 1 m \left(1-\frac 1 m\right)}\right)^2.
$$
\end{theorem}
\begin{proof}
Let $t >1$ and set $\xx=\epsilon t\ee$ and $t_i=t$ for $1\leq i \leq m$.  Then 
$\xx+t_i \ee_i = t(\epsilon \ee+ \ee_i) \in \Lambda_+$ since 
$$
h[\vv_1,\ldots, \vv_m](\epsilon \ee + \ee_i)= \epsilon h(\ee)- D_{\vv_i}h(\ee) = h(\ee)(\epsilon -\tr(\vv_i)) \geq 0. 
$$
Apply Corollary \ref{corbond} to conclude that for each $t>1$:
$$
\left(\epsilon t+ \left(1-\frac 1 m\right)\frac t {t-1}\right) \ee + \left(1-\frac 1 m + \frac t m \right) \one \in \Gamma_+. 
$$

Hence by (the homogeneity of $\Gamma_+$ and) Remark \ref{hypid}, the maximal zero is at most
$$
\inf \left\{ \frac {\epsilon t+ \left(1-\frac 1 m\right)\frac t {t-1}} {1-\frac 1 m + \frac t m }  : t >1\right\}.
$$
It is a simple exercise to deduce that the infimum is exactly what is displayed in the statement of the theorem. 
\end{proof}

\section{Proof of Theorem \ref{t1}}
To prove Theorem \ref{t1} we use the following theorem which for $h=\det$ appears in \cite{MSS1,MSS2}:
\begin{theorem}\label{hypprob}
Suppose $h$ is hyperbolic  with respect to $\ee$. Let $\XR_1, \ldots, \XR_m$ be independent random vectors in $\Lambda_+$ of rank at most one and  with finite supports  such that 
\begin{equation}\label{hypeta2}
\EE\sum_{i=1}^m \XR_i =\ee, 
\end{equation}
and 
\begin{equation}\label{hyptr}
\tr(\EE \XR_i) \leq \epsilon \mbox{ for all } 1\leq i \leq m.
\end{equation}
Then 
\begin{equation}\label{hypbig}
\PP\left[ \lma\left(\sum_{i=1}^m \XR_i \right) \leq \delta(\epsilon,m) \right] >0.
\end{equation}

\begin{proof}
Let $V_i$ be the support of $\XR_i$, for each $1 \leq i \leq m$. By Theorem \ref{mixedchar}, the family 
$$
\{h(t\ee- \vv_1-\cdots-\vv_m)\}_{\vv_i \in V_i}
$$
is compatible. By Theorem \ref{expfam} there are vectors $\vv_i \in V_i$, $1\leq i \leq m$, such that the largest zero of $h(t\ee- \vv_1-\ldots-\vv_m)$ is smaller or equal to the largest zero of 
$$
\EE h(t\ee- \XR_1-\cdots-\XR_m)= \EE h[\XR_1,\ldots, \XR_m](t\ee+\one)= h[\EE \XR_1,\ldots, \EE \XR_m](t\ee+\one).
$$
The theorem now follows from Theorem \ref{mainbound}. 
\end{proof}

\end{theorem}

\begin{proof}[Proof of Theorem \ref{t1}]
For $1\leq i \leq k$, let $\xx^i=(x_{i1},\ldots,x_{in})$ where $\yy=\{x_{ij} : 1\leq i \leq k, 1\leq j \leq k\}$ are independent variables. Consider the polynomial 
$$
g(\yy) = h(\xx^1)h(\xx^2) \cdots h(\xx^k) \in \RR[\yy],
$$
which is hyperbolic with respect to $\ee^1\oplus \cdots \oplus \ee^k$, where $\ee^i$ is a copy of $\ee$ in the variables $\xx^i$, for all $1 \leq i \leq k$. The hyperbolicity cone of $g$ is the direct sum $\Lambda_+:=\Lambda_+(\ee^1) \oplus \cdots \oplus \Lambda_+(\ee^k)$, where $\Lambda_+(\ee^i)$ is a copy of $\Lambda_+(\ee)$ in the variables $\xx^i$,   for all $1 \leq i \leq k$. 

Let $\XR_1, \ldots, \XR_m$ be independent random vectors in $\Lambda_+$  such that for all $1\leq i \leq k$ and  $1\leq j \leq m$:
$$
\PP\left[ \XR_j = k\uu_j^i\right] =  \frac 1 k,
$$
where $\uu_1^i, \ldots, \uu_m^i$ are copies in $\Lambda_+(\ee^i)$ of $\uu_1, \ldots, \uu_m$. Then 
\begin{align*}
\EE \XR_j &= \uu_j^1 \oplus \uu_j^2 \oplus \cdots \oplus \uu_j^k, \\
\tr(\EE \XR_j) &= k\tr(\uu_j) \leq k\epsilon, \mbox{ and } \\
\EE \sum_{j=1}^m \XR_j &= \ee^1\oplus \cdots \oplus \ee^k,
\end{align*}
for all $1\leq j \leq k$. By Theorem \ref{hypprob} there is a partition 
$S_1\cup \cdots \cup S_k =[m]$ such that 
$$
\lma\left(\sum_{i \in S_1}k\uu_i^1+\cdots + \sum_{i \in S_k}k\uu_i^k  \right)\leq \delta(k\epsilon,m). 
$$
However 
$$
 \lma\! \left(\sum_{i \in S_1}k\uu_i^1+\cdots + \sum_{i \in S_k}k\uu_i^k \right) = k \! \max_{1\leq j \leq k} \lma \! \left(\sum_{i \in S_j}\uu_i^j \right) = 
 k \! \max_{1\leq j \leq k} \lma \! \left(\sum_{i \in S_j}\uu_i \right), 
$$
and the theorem follows.
\end{proof}

\section{On a conjecture on the optimal bound}
We have seen that the core of the proof of Theorem \ref{t1} is to bound the zeros of mixed characteristic polynomials. To achieve better bounds in Theorem \ref{t1} we are therefore motivated to look closer at the following problem. 
\begin{problem}\label{central}
Let $h$ be a polynomial of degree $d$ which is hyperbolic with respect to $\ee$, and let $\epsilon >0$ and $m \in \ZZ_+$ be given. Determine the largest possible maximal zero, $\rho=\rho(h,\ee,\epsilon,m)$, of mixed characteristic polynomials:
$$
\chi[\vv_1,\ldots, \vv_m](t):=h[\vv_1, \ldots, \vv_m](t\ee + \one) 
$$
subject to the conditions 
\begin{enumerate}
\item $\vv_1, \ldots, \vv_m \in \Lambda_+$, 
\item $\vv_1 + \cdots + \vv_m = \ee$, and
\item $\tr(\vv_i) \leq \epsilon$ for all $1 \leq i \leq m$. 
\end{enumerate}
\end{problem}
The following conjecture was made by Marcus \emph{et al.} \cite{MSS2} in the case when $h = \det$, but we take the liberty to extend the conjecture to any hyperbolic polynomial. 
\begin{conjecture}\label{maxmax}
The maximal zero in Problem \ref{central} is achieved for 
$$
\vv_1=\cdots =\vv_k= \frac \epsilon d \ee, \vv_{k+1}= \left(1-\frac k d \epsilon\right)\ee, \vv_{k+2}=\vv_{k+3}= \cdots= \vv_{m}=0, 
$$
where $k= \lfloor d/\epsilon \rfloor$. 
\end{conjecture}
We will prove here that Conjecture \ref{maxmax} is equivalent to the following seemingly weaker conjecture. 

\begin{conjecture}\label{maxmax2}
The maximal zero in Problem \ref{central} is achieved for some $\vv_1, \ldots, \vv_m$ where $\vv_i \in \Lambda_{++}\cup \{0\}$ for each $i \in [m]$. 
\end{conjecture}

We start by proving that there is a solution to Problem \ref{central} for which the $\vv_i$'s have correct traces, i.e., as those in Conjecture \ref{maxmax}.  By a ``solution" to Problem \ref{central}  we mean a list of vectors $\vv_1, \ldots, \vv_m$, as in Problem \ref{central}, which realize the maximal zero. First a useful lemma.

\begin{lemma}\label{nicein}
Suppose $\uu, \vv, \ww \in \Lambda_+$. Then 
$$
(D_\uu D_\vv h(\ww))^2 \geq D_\uu^2 h(\ww) \cdot D_\vv^2 h(\ww), 
$$
and hence 
$$
D_\uu D_\vv h(\ww) \geq \min\{ D_\uu^2 h(\ww), D_\vv^2 h(\ww)\}.
$$
\end{lemma}
\begin{proof}
By continuity we may assume  $\uu, \vv, \ww \in \Lambda_{++}$. Then 
the polynomial 
\begin{align*}
& g(x,y,z):=h(x\uu+y\vv+z\ww) = h(\ww)z^d+ \big(D_\uu h(\ww) x + D_\vv h(\ww)y\big)z^{d-1}+ \\ 
&+ \left(  D_\uu D_\vv h(\ww) xy + \frac 1 2 D_\uu^2 h(\ww) x^2 + \frac 1 2 D_\vv^2 h(\ww)y^2\right)z^{d-2}+ \cdots
\end{align*}
is hyperbolic with hyperbolicity cone containing the positive orthant. By Theorem \ref{direct} (1) so is $\partial^{d-2} g /\partial z^{d-2}$, and hence the polynomial 
$$
2\frac {\partial^{d-2} g} {\partial z^{d-2}} \big( (1,0,0)+ t(0,0,1) \big) =   D_\uu^2 h(\ww)  + 2D_\uu D_\vv h(\ww)t  + D_\vv^2 h(\ww)t^2
$$
is real--rooted. Thus its discriminant is nonnegative, which yields the desired inequality. 
\end{proof}

\begin{proposition}\label{righttrace}
There is a solution to Problem \ref{central} such that all but at most one of the $\vv_i$'s have trace either zero or $\epsilon$. 

Moreover, if there is a solution to Problem \ref{central} which satisfies the condition in Conjecture \ref{maxmax2}, then there is such a solution such that all but at most one of the $\vv_i$'s have trace either zero or $\epsilon$. 
\end{proposition}
\begin{proof}
Let $\vv_1, \ldots, \vv_m$ be a solution to Problem \ref{central}, and let $\rho$ be the maximal zero. Suppose $0<\tr(\vv_1), \tr(\vv_2) <\epsilon$. By Remark \ref{hypid} $\rho \ee + \one$ is in the hyperbolicity cone $\Gamma_+$ of $h[\vv_1, \ldots, \vv_m]$. Since also $-\ee_1, -\ee_2 \in \Gamma_+$ we have $\ww:= \rho\ee+\one-\ee_1-\ee_2 \in \Gamma_+$, and hence $\ww$ is in the (closed) hyperbolicity cone of $g= h[\vv_3, \ldots, \vv_m]$. By Lemma~\ref{nicein} we may assume 
$$
D_{\vv_1}D_{\vv_2} g(\ww) \geq D_{\vv_1}^2g(\ww) \geq 0,
$$
since otherwise change the indices $1$ and $2$. For  
$$
0 \leq s \leq \min\left\{1, \frac {\epsilon -\tr(\vv_2)} {\tr(\vv_1)}\right\}, 
$$ we have (since $(1-D_{\vv_1})(1-D_{\vv_2})g(\ww)=0$):
\begin{align*}
& h[\vv_1-s\vv_1, \vv_2+s\vv_1, \vv_3, \ldots, \vv_m](\rho \ee +\one) \\
=& -s(D_{\vv_1}D_{\vv_2} g(\ww)-D_{\vv_1}^2g(\ww))-s^2D_{\vv_1}^2g(\ww) \leq 0. 
\end{align*}
Hence the maximal zero of $\chi[\vv_1-s\vv_1, \vv_2+s\vv_1, \vv_3, \ldots, \vv_m](t)$  is at least $\rho$, and since $\rho$ is the largest possible maximal zero 
$$ 
\chi[\vv_1-s\vv_1, \vv_2+s\vv_1, \vv_3, \ldots, \vv_m](\rho)=0.
$$
We may therefore alter $\vv_1, \vv_2$ so that either $\vv_1=0$ or $\tr(\vv_2) = \epsilon$, while retaining the maximal zero $\rho$. Continuing this process we arrive at a solution of the desired form. 
\end{proof}

\begin{lemma}\label{average}
Suppose $\vv_1,\ldots, \vv_m$ is a solution to Problem \ref{central} such that 
$\tr(\vv_1)= \tr(\vv_2)$ and $\vv_1, \vv_2 \in \Lambda_{++}$. Then 
$(\vv_1+\vv_2)/2, (\vv_1+\vv_2)/2, \vv_3, \ldots, \vv_m$ is also a solution to Problem \ref{central}. 
\end{lemma}

\begin{proof}
Let $\vv_1(s)= (1-s)\vv_1+s\vv_2$ and $\vv_2(s)= (1-s)\vv_2+s\vv_1$. Then $\tr(\vv_1(s))=\tr(\vv_2(s)) = \tr(\vv_1)$, $\vv_1(s)+\vv_2(s)=\vv_1+\vv_2$, and $\vv_1(s), \vv_2(s) \in \Lambda_{++}$ for all 
$s \in (-\delta, 1+\delta)$ for some $\delta>0$. Let $\rho$ be the maximal zero in Problem \ref{central}. Then 
the function 
$$
(-\delta, 1+\delta) \ni s \mapsto \chi[\vv_1(s),\vv_2(s), \vv_3, \ldots, \vv_m](\rho) 
$$
is a degree at most two polynomial which has local minima at $s=0$ and $s=1$. Hence this function is identically zero, and thus 
$$
\chi[\vv_1(1/2),\vv_2(1/2), \vv_3, \ldots, \vv_m](\rho) =0 
$$
as desired. 
\end{proof}
\begin{remark}\label{infav} 
Suppose $\vv_1, \ldots, \vv_m$ is a solution to Problem \ref{central} such that $\vv_1, \ldots, \vv_k \in \Lambda_{++}$ all have the same trace, and let 
$$
\vv = \frac 1 k (\vv_1 + \cdots + \vv_k).
$$
By applying Lemma \ref{average} infinitely many times (and invoking Hurwitz' theorem on the continuity of zeros) we see that also $\vv, \ldots, \vv, \vv_{k+1}, \ldots, \vv_m$ is a solution to Problem \ref{central}. 
\end{remark}

Clearly Conjecture \ref{maxmax} implies Conjecture \ref{maxmax2}. To prove the other implication assume Conjecture \ref{maxmax2}. Then by Proposition \ref{righttrace} and Remark \ref{infav} we may assume that we have a solution of the form $\vv_1, \ldots, \vv_m$, where 
$\vv_1=\cdots=\vv_k=\vv$, $\vv_{k+1}=\ee-k\vv$, $\vv_{k+2}=\cdots=\vv_m=0$, and where 
$\vv, \ee-k\vv \in \Lambda_{++}$ and $\tr(\vv)= \epsilon$ and $0<d-k\epsilon = \tr(\ee-k\vv)<\epsilon$. Hence we want to maximize the largest zero of 
\begin{equation}\label{gv}
g_\vv(t):=(1-D_\vv)^{k} (1-D_\ee+kD_\vv) h(t\ee)
\end{equation}
where 
\begin{itemize}
\item[(a)] $\vv, \ee-k\vv \in \Lambda_{++}$
\item[(b)] $\tr(\vv)=\epsilon$, where $0<d-k\epsilon <\epsilon$.
\end{itemize}
Let $I \subseteq \RR$ be an interval. We say that a univariate polynomial is $I$--\emph{rooted} if all its zeros lie in $I$. 
\begin{lemma}\label{tgv}
Let $g_\vv(t)$ be given by \eqref{gv}. Then $g_\vv(t)= T_{k,d}(h(t\ee-\vv))$ where $T_{k,d}: \RR[t] \rightarrow \RR[t]$ is the linear operator defined by 
$$
T_{k,d}\left(\sum_{j\geq 0} a_jt^j\right) = -\sum_{j=0}^d \left( \frac {j+1} {k+1} a_{j+1} +(d-1-j)a_j\right) (d-j)! \binom {k+1}{d-j}t^j.
$$
Moreover if $f$ is a $[0,1/k]$--rooted polynomial of degree $d$, then $T_{k,d}(f)$ is real--rooted. 
\end{lemma}
\begin{proof}
By \eqref{mag} 
\begin{equation}\label{collect}
h(t\ee-\vv)= \sum_{j=0}^d (-1)^j \frac 1 {j!} D_\vv^jh(\ee)t^{d-j}=:\sum_{j \geq 0} a_j t^j.
\end{equation}
Note that $D_\vv^j h(t\ee)= D_\vv^j h(\ee)t^{d-j}$ and $D_\vv^j D_\ee h(t\ee)= D_\ee D_\vv^j h(t\ee)= (d-j)D_\vv^j h(\ee)t^{d-j-1}$. Expanding \eqref{gv} and comparing coefficients with \eqref{collect} one sees that $g_\vv(t)=T_{k,d}(h(t\ee-\vv)$. 

To prove the final statement of the lemma we may by Hurwitz' theorem on the continuity of zeros assume that $f$ is a $(0,1/k)$--rooted polynomial of degree $d$. We may choose a hyperbolic degree $d$ polynomial $h$ and a vector $\vv$ such that $f(t)= h(t\ee -\vv)$, for example 
$h(x,y)= (-y)^df(-x/y)$, $\ee=(1,0)$ and $\vv=(0,1)$. Then $\vv \in \Lambda_{++}$ and $\ww=\ee-k\vv \in \Lambda_{++}$ by e.g. \eqref{dilambdas}. Hence 
$$
T_{k,d}(f)(t) = \chi[\vv,\vv, \ldots, \vv, \ww](t)
$$
is real--rooted. 
\end{proof}

The \emph{trace}, $\tr(f)$, of a non-constant polynomial is the sum of the the zeros of $f$ (counted with multiplicity). Let $\MR_d$ be the affine space of all monic real polynomials of degree $d$. 
\begin{lemma}\label{affine}
Let $T : \MR_d \rightarrow \MR_m$ be an affine linear operator, and let $\epsilon>0$. Suppose $T$ sends $[a,b]$--rooted polynomials to real--rooted polynomials. Consider the problem of maximizing the largest zero of $T(f)$ over all $[a,b]$--rooted polynomials $f \in \MR_d$ with $\tr(f) = \epsilon$.  Then this (maximal) zero is   achieved for some $T(f)$, where $f$ has at most one distinct zero in $(a,b)$. 

Moreover, if the maximal zero above is achieved for some $T(f)$, where $f \in \MR_d$  is $(a,b)$--rooted, then the maximal zero is also achieved for $T ((t-\epsilon/d)^d)$.   
\end{lemma}
\begin{proof}
Let $\A=\A(a,b,d,\epsilon)$ be the set of all $[a,b]$--rooted  polynomials $f \in \MR_d$ with $\tr(f)=\epsilon$. 
Note that continuity, compactness and Hurwitz' theorem the maximum zero (say $\rho$) is achieved for some $T(f)$, where $f \in \A$. 
 We argue that we may move zeros of $f$ to the boundary of $[a,b]$, while retaining $\tr(f)$ and the maximal zero of $T(f)$ as long as $f$ has at least two distinct zeros in $(a,b)$.  
 
 Suppose $a<\alpha<\beta<b$ are two zeros of $f\in \A$ and that the maximal zero is realized for  $T(f)$.  
For $0<|s| \leq \min(b-\beta, \alpha-a, \beta-\alpha)$, let 
$$
f_s(x) := \frac {(x-\alpha-s)(x-\beta+s)}{(x-\alpha)(x-\beta)} f(x), 
$$
and note that $f_s \in \A$ and 
$$
f= (1-\theta)f_{s}+ \theta f_{-s}, \quad \mbox{ where } \quad \theta = \frac 1 2 \left(1- \frac s {\beta-\alpha}\right)\in [0,1].
$$
By assumption $T(f_s)(\rho) \geq 0$. Since $0=T(f)(\rho)= (1-\theta)T(f_{s})(\rho)+ \theta T(f_{-s})(\rho)$, 
we conclude that $T(f_s)(\rho) = T(f_{-s})(\rho)=0$.  Hence the maximal zero $\rho$ is realized also for $T(f_s)$ where 
$s= -\min(b-\beta, \alpha-a, \beta-\alpha)$. By possible iterating this process a few times we will have moved at least one interior zero to the boundary. We can continue until there is at most one distinct zero in $(a,b)$.  

Suppose the maximal zero $\rho$ above is achieved for some  $f \in \MR_d$ which is $(a,b)$--rooted. Then  $\rho$ is also attained for the same problem when we replace $[a,b]$ by $[r,s]$ where $a<r<s<b$ and $r-a$ and $b-s$ are sufficiently small. Hence, by what we have just proved, for each such $r,s$ there are nonnegative integers $i,j$ with $i+j \leq d$ such that 
\begin{equation}\label{abs}
T\left( (t-r)^i (t-s)^j \left(t - \frac {\epsilon -ir-js}{d-i-j}\right)^{d-i-j}\right) (\rho)=0. 
\end{equation}
The left--hand--side of \eqref{abs} is a polynomial, say $P_{i,j}(r,s) \in \RR[r,s]$. Hence the polynomial 
$\prod_{i,j}P_{i,j}(r,s)$, where the product is over all $i,j$ which are realized for some such $r,s$,  vanishes on a set with nonempty interior, so it is identically zero. Hence $P_{i,j}(r,s) \equiv 0$ for some $i,j$. But then 
$0=P_{i,j}(\epsilon/d, \epsilon/d)=T((t-\epsilon/d)^d)(\rho)$ as desired. 
\end{proof}

We may now finish the proof of that Conjecture \ref{maxmax2} implies Conjecture \ref{maxmax}. It remains to prove that the largest possible zero of $g_\vv(t)$, where $\vv$ satisfies (a) and (b) above is achieved when 
$h(t\ee-\vv)= (t-\epsilon/d)^d$, assuming (we assume Conjecture \ref{maxmax2})  that the maximum is achieved for some $\vv$ where $h(t\ee-\vv)$ is $(0,1/k)$--rooted. By e.g. considering $h=\det$ on $d \times d$--matrices this is equivalent to proving that the maximal zero of $T_{k,d}(f)$ where $f$ ranges over all monic $[0,1/k]$--rooted polynomials of degree $d$ with trace $\epsilon$ is achieved when $f=(t-\epsilon/d)^d$, under the assumption that the maximal zero is achieved for some $(0,1/k)$--rooted $f$. This follows from the last part of Lemma \ref{affine}.

\section{Sharpness of the bound in Theorem \ref{t1}}\label{sbound}
We will in this section use results known about the asymptotic behavior of the largest zero of Jacobi polynomials to see that the bound in  Theorem \ref{t1} is close to being optimal. 

Consider the degree $d$ elementary symmetric polynomial in $mk$ variables: 
$$
e_d(x_1,\ldots, x_{mk}) = \sum_{|S|=d} \prod_{i \in S}x_i, 
$$
which is hyperbolic with respect to the all ones vector $\one \in \RR^{mk}$, see e.g. \cite{BrOp,COSW}. Since the  coefficients of $e_d(\xx)$ are nonnegative, its hyperbolicity cone contains the positive orthant. If 
$\ee_i$ denotes the $i$th standard basis vector, then 
$$\tr(\ee_i) = \frac {d} {mk}, \ \ \rk(\ee_i)=1 \ \ \mbox{ and } \ \ \ee_1+\cdots+\ee_{mk}=\one, $$ 
for all $1\leq i \leq mk$. 
By symmetry, the partition 
$$
S_1=\{1,\ldots, m\}, S_2=\{m+1, \ldots, 2m\}, \ldots, S_k=\{(k-1)m+1,\ldots, km\}
$$
minimizes the bound in \eqref{sqbound}. Now 
\begin{align*}
e_d\left(t\one-\sum_{i \in S_1}\ee_i\right) &= \sum_{|A|=d} (t- 1)^{|A\cap S_1|}t^{d- |A\cap S_1|}\\
&= \sum_{j=0}^d  \binom {m(k-1)} {j} \binom m {d-j}(t-1)^{d-j}t^{j}\\
&= P^{(mk-m-d,m-d)}_d(2t-1),
\end{align*}
where $P^{(\alpha,\beta)}_k(t)$ is a Jacobi polynomial. The asymptotic behavior  of the largest zero of Jacobi polynomials is well studied, see e.g. \cite{Is, Kra}. For example, if $\alpha_d, \beta_d >-1$ satisfy 
$$
\frac {\alpha_d}{\alpha_d +\beta_d +2d} \to a \mbox{ and } \frac {\beta_d}{\alpha_d +\beta_d +2d} \to b \mbox{ as } d \to \infty, 
$$
then the largest zero of $P_d^{(\alpha_d,\beta_d)}(t)$ converges to 
\begin{equation}\label{abby}
b^2-a^2+\sqrt{(a^2+b^2-1)^2-4a^2b^2},
\end{equation}
as $d \to \infty$, see \cite[Theorem 8]{Is}. 

Fix $\epsilon$ and $k$, and let $m:=m(d)=\lceil d/(\epsilon k) \rceil $ and $\alpha_d=mk-m-d$, $\beta_d=m-d$.  Then $a=1-1/k-\epsilon$ and $b=1/k-\epsilon$, and so by \eqref{abby} the largest zero of $P^{(\alpha_d, \beta_d)}_d(2t-1)$ converges to 
$$
\frac 1 k + \epsilon \frac {k-2} k + 2 \frac{\sqrt{k-1}}{k}  \sqrt{\epsilon-\epsilon^2}, 
$$
which should be compared to the bound achieved by Theorem \ref{t1} (as $m\to \infty$): 
$$
\frac 1 k + \epsilon + 2\frac {\sqrt{k}} k \sqrt{\epsilon}.
$$
We conclude: 
\begin{proposition}\label{lowprop}
There is no version of Theorem \ref{t1} with an ($m,d$-independent) bound in the right--hand--side of \eqref{sqbound} which is smaller than 
\begin{equation}\label{bbound}
\frac 1 k + \epsilon \frac {k-2} k + 2 \frac{\sqrt{k-1}}{k}  \sqrt{\epsilon-\epsilon^2}, 
\end{equation}
for $\epsilon \leq 1-1/k$. 
\end{proposition}

\begin{remark}
It is known that if $1<d<n-1$, then $e_d(x_1, \ldots, x_n)$ is \emph{not} a determinantal polynomial, i.e., there is no tuple of positive semidefinite matrices $A_1, \ldots, A_n$ such that 
$$
e_d(x_1, \ldots, x_n) = \det(x_1A_1+ \cdots+x_nA_n).
$$
Thus we cannot directly derive an analog of Proposition~\ref{lowprop} for Theorem~\ref{MSSmain}. 
\end{remark}

\section{Consequences for strong Rayleigh measures and weak half-plane property matroids}
A discrete probability measure, $\mu$, on $2^{[n]}$ is called \emph{strong Rayleigh} if its \emph{multivariate partition function}
$$
P_\mu(\xx) := \sum_{S \subseteq [n]} \mu(\{S\}) \prod_{j \in S}x_j,
$$ 
is \emph{stable}, i.e., if $P_\mu(\xx) \neq 0$ whenever $\Im(x_j)>0$ for all $1 \leq j \leq n$. Strong Rayleigh measures were investigated in \cite{BBL}, see also \cite{Pem,Wag}. We shall now reformulate Theorem \ref{t1} in terms of strong Rayleigh measures. The measure $\mu$ is of \emph{constant sum} $d$ if $|S|=d$ whenever $\mu(\{S\}) \neq 0$, i.e., if $P_\mu(\xx)$ is homogeneous of degree $d$. It is not hard to see that a constant sum measure $\mu$ is strong Rayleigh if and only if $P_\mu(\xx)$ is hyperbolic with respect to the all ones vector $\one$ and $\RR_+^n \subseteq \Lambda_+(\one)$, see \cite{BBL}. Note that if $\ee_i$ is the $i$th standard basis vector then 
$$
\tr(\ee_i) = \sum_{S \ni i} \mu(\{S\})= \PP[S : i \in S], 
$$ 
where the trace is defined as in the introduction for the hyperbolic polynomial $P_\mu$, with $\ee=\one$. 
If $S \subseteq [n]$ we write $\ee_S:=\sum_{i \in S}\ee_i$. The following theorem is now an immediate consequence of Theorem~\ref{t1}. 
\begin{theorem}\label{t11}
Let $k\geq 2$ be an integer and $\epsilon$ a positive real number. 
Suppose $\mu$ is  a constant sum strong Rayleigh measure on $2^{[n]}$ such that $\PP[S : i \in S] \leq \epsilon$ for all $1\leq i \leq n$. Then there is a partition $S_1 \cup \cdots \cup S_k=[n]$ such that 
$$
\| e_{S_i} \|  = \lma(\ee_{S_i}) \leq \frac 1 k \delta(k\epsilon,n)
$$
for each $1 \leq i \leq n$. 
\end{theorem}
Let us also see that Theorem \ref{t1} easily follows from Theorem \ref{t11}. Assume the hypothesis in Theorem~\ref{t1}, and form the polynomial 
$$
P(\xx)= h(x_1\uu_1+\cdots+x_m\uu_m)/h(\ee).
$$
It follows that $P(\xx)$ is hyperbolic with hyperbolicity cone containing the positive orthant. Since $\rk(\uu_i) \leq 1$ for all 
$1\leq i \leq m$ we may expand $P(\xx)$ as 
$$
P(\xx)= \sum_{S \subseteq [m]} \mu(\{S\}) \prod_{j \in S}x_j,
$$
where $\mu(\{S\}) \geq 0$ for all $S \subseteq [m]$. Since $\tr_h(\uu_i)= \tr_P(\ee_i)$ for all 
$1\leq i \leq m$, the conclusion in Theorem~\ref{t1} now follows from Theorem~\ref{t11}.

The \emph{support} of $\mu$ is $\{S: \mu(\{S\})>0\}$. 
Choe \emph{et al.} \cite{COSW} proved that the support of a constant sum strong Rayleigh measure is the set of bases of matroid. Such matroids are called \emph{weak half-plane property matroids}. The rank function, $r$, of such a matroid is given by 
$$
r(S) = \rk\left(\sum_{i \in S}\ee_i\right),
$$
where $\rk$ is the rank function associated to the hyperbolic polynomial $P_\mu$ as defined in the introduction, see \cite{BrObs,Gu}. Edmonds Base Packing Theorem \cite{Edm} characterizes, in terms of the rank function, when a matroid contains $k$ disjoint bases. Namely if and only if 
$$
r(S) \geq d -\frac {n-|S|} k, \quad \mbox{ for all } S \subseteq [n],
$$ 
where $r$ is the rank function of a rank $d$ matroid on $n$ elements. Using Theorem~\ref{t11} we may deduce a sufficient condition (of a totally different form) for a matroid with the weak half-plane property to have $k$ disjoint bases:
\begin{theorem}\label{packing}
 Let $k\geq 2$ be an integer. 
Suppose $\mu$ is  a constant sum strong Rayleigh measure such that 
$$
\PP[S : i \in S] \leq \left(\frac 1 {\sqrt{k-1}} - \frac 1 {\sqrt{k}}\right)^2
$$ for all $1\leq i \leq n$. Then the support of $\mu$ contains $k$ disjoint bases. 
 \end{theorem}
 
\begin{proof}
Suppose $\tr(\ee_i) \leq \epsilon$ for all $i$. Let $S_1 \cup \cdots \cup S_k=[n]$ be a partition afforded by Theorem \ref{t11}, and let 
$\vv_j= \sum_{i \in S_j}\ee_i$ for each $j \in [n]$. If we can prove that $\lmi(\vv_j)>0$, then $\rk(\vv_j)=\rk(\one)$ and so  $S_j$ contains a basis. Now, by \eqref{dilambdas}, Theorem \ref{t11}, and the convexity of $\lma$:
\begin{align*}
\lmi(\vv_j) &= 1-\lma(\one-\vv_j) =1-\lma\left(\sum_{i \neq j}\vv_i\right)  \\
&\geq 1- \sum_{i \neq j}\lma(\vv_i) \geq 1-\frac {k-1} {k} \delta(k\epsilon,n) \\
&> 1- \frac {k-1} {k} \left(1+\sqrt{\epsilon k}\right)^2.
\end{align*}
 Hence we want the quantity on the left hand side to be nonnegative, which is equivalent to 
$$
\epsilon \leq \left(\frac 1 {\sqrt{k-1}} - \frac 1 {\sqrt{k}}\right)^2.
$$
\end{proof}
We have not investigated the sharpness of Theorem~\ref{packing}, nor if it is possible to prove analogous versions for arbitrary matroids. For an arbitrary matroid on $[n]$ one could take the uniform measure on the set of bases of the matroid and define 
$\tr(i) = \PP[S: i\in S]$.  What trace bounds guarantees the existence of $k$ disjoint bases? 

It would be interesting to see if other theorems on matroids have analogs for weak half-plane property matroids using Theorem~\ref{t11}. Also, can we find continuous versions of theorems in matroid theory using the analogy that Theorem~\ref{t1} can be seen as a continuous version of Edmonds Base Packing Theorem?


\begin{thebibliography}{99}
\bibitem{ABG} M.~F.~Atiyah, R.~Bott, L~G\aa rding, Lacunas for hyperbolic differential operators with constant coefficients. I, Acta Math. {\bf 124} (1970), 109--189.


\bibitem{BGLS} H.~H.~Bauschke, O.~G\"uler, A.~S.~Lewis, H.~S.~Sendov,  Hyperbolic polynomials and convex analysis, Canad. J. Math. {\bf 53} (2001), 470--488.

\bibitem{BBL} J.~Borcea, P.~Br\"and\'en,  T.~M.~Liggett, {Negative dependence and the geometry of polynomials}, J. Amer.
Math. Soc. {\bf 22} (2009), 521--567, \url{http://arxiv.org/abs/0707.2340}.

\bibitem{BrObs} P.~Br\"and\'en, { Obstructions to determinantal representability,} Adv. Math., {\bf 226} (2011), 1202--1212, \url{http://arxiv.org/pdf/1004.1382.pdf}. 
%
\bibitem{BrOp} P.~Br\"and\'en, Hyperbolicity cones of elementary symmetric polynomials are spectrahedral, Optim. Lett. {\bf 8} (2014), 1773--1782, \url{http://arxiv.org/abs/1204.2997}.

%
%




\bibitem{COSW} 
Y.~Choe, J.~Oxley, A.~Sokal, D.~G.~Wagner, { Homogeneous multivariate 
polynomials with the half-plane property}. 
Adv. Appl. Math. {\bf 32} (2004), 88--187, \url{http://arxiv.org/pdf/math/0202034.pdf}.

\bibitem{Cas} P.~G.~Casazza, Consequences of the Marcus/Spielman/Srivastava solution to the Kadison--Singer Problem, \url{http://arxiv.org/abs/1407.4768}.

\bibitem{CS} M.~Chudnovsky, P.~Seymour, The roots of the independence polynomial of a clawfree graph, J. Combin. Theory Ser. B {\bf 97} (2007),  350--357.

\bibitem{Edm} J.~Edmonds, 
Lehman's switching game and a theorem of Tutte and Nash-Williams, 
J. Res. Nat. Bur. Standards Sect. B {\bf 69 B}  (1965),  73--77. 

\bibitem{Ga} 
L.~G\aa rding, {An inequality for hyperbolic polynomials}, 
J. Math. Mech. {\bf 8} (1959), 957--965. 

\bibitem{Gu} L.~Gurvits, Combinatorial and algorithmic aspects of hyperbolic polynomials, \url{http://arxiv.org/abs/math/0404474}.
%
%
%
%

\bibitem{Horm} L.~H\"ormander, The analysis of linear partial differential operators. II. Differential operators with constant coefficients, Springer-Verlag, Berlin, 1983.

\bibitem{Is} M.~E.~H.~Ismael, X.~Li, 
Bound on the extreme zeros of orthogonal polynomials, 
Proc. Amer. Math. Soc. {\bf 115} (1992), 131--140. 

\bibitem{KS} R.~V.~Kadison, I.~M.~Singer, Extensions of pure states, Amer. J. Math. {\bf 81} (1959),  383--400.

\bibitem{Kra} I.~Krasikov, 
On extreme zeros of classical orthogonal polynomials,
J. Comput. Appl. Math. {\bf 193} (2006), 168--182. 

%


\bibitem{MSS1} A.~W.~Marcus, D.~A.~Spielman, N.~Srivastava, Interlacing families I: Bipartite Ramanujan graphs of all degrees, Ann. of Math. (to appear), \url{http://arxiv.org/abs/1304.4132 }.

\bibitem{MSS2} A.~W.~Marcus, D.~A.~Spielman, N.~Srivastava, Interlacing families II: Mixed characteristic polynomials and the Kadison-Singer problem, Ann. of Math. (to appear), \url{http://arxiv.org/abs/1306.3969}.

\bibitem{Pem} R.~Pemantle, Hyperbolicity and stable polynomials in combinatorics and probability, Current developments in mathematics, 2011, 57--123, Int. Press, Somerville, MA, 2012, \url{http://arxiv.org/abs/1210.3231}. 

\bibitem{Ren} J.~Renegar,  Hyperbolic programs, and their derivative relaxations,  Found. Comput. Math., 
{\bf 6} (2006), 59--79.

\bibitem{Vin} V.~Vinnikov, 
LMI representations of convex semialgebraic sets and determinantal representations of algebraic hypersurfaces: past, present, and future,  Mathematical methods in systems, optimization, and control, 325--349, 
Oper. Theory Adv. Appl., {\bf 222}, BirkhŠuser/Springer Basel AG, Basel, 2012, \url{http://arxiv.org/abs/1205.2286}. 

\bibitem{Wag} D.~G.~Wagner, Multivariate stable polynomials: theory and applications, Bull. Amer. Math. Soc. {\bf 48} (2011), 53--84, \url{http://arxiv.org/abs/0911.3569}. 

\bibitem{We} N.~Weaver, 
The Kadison-Singer problem in discrepancy theory,
Discrete Math. {\bf 278} (2004),  227--239.

\end{thebibliography}
\end{document}